\documentclass[11pt,reqno]{amsart}

\usepackage[T1]{fontenc}
\usepackage[utf8]{inputenc}
\usepackage{amsmath,amssymb,amsfonts,amsthm,mathtools}
\usepackage{booktabs}
\usepackage{microtype}
\usepackage[hidelinks]{hyperref}
\hypersetup{
  pdftitle={The Complete Cubic Walsh Spectrum of a Permutation-Inverse Boolean Family},
  pdfauthor={Kaimin Cheng},
  pdfkeywords={finite fields, Boolean functions, Walsh spectra, plateaued functions, Kasami APN functions}
}

\makeatletter
\@namedef{subjclassname@2020}{\textup{2020} Mathematics Subject Classification}
\makeatother

\theoremstyle{plain}
\newtheorem{theorem}{Theorem}[section]
\newtheorem{lemma}[theorem]{Lemma}
\newtheorem{proposition}[theorem]{Proposition}
\newtheorem{corollary}[theorem]{Corollary}

\theoremstyle{remark}
\newtheorem{remark}[theorem]{Remark}

\newcommand{\Fq}{\mathbb F_q}
\newcommand{\Fqq}{\mathbb F_{q^2}}
\newcommand{\Ftwo}{\mathbb F_2}
\newcommand{\F}{\mathbb F}
\newcommand{\Tr}{\operatorname{Tr}}

\newcommand{\abs}[1]{\left|#1\right|}
\newcommand{\set}[1]{\left\{#1\right\}}
\newcommand{\W}{W}
\newcommand{\chiq}[1]{\chi_q\!\left(#1\right)}

\title[Complete cubic Walsh spectrum]{The Complete Cubic Walsh Spectrum of a Permutation-Inverse Boolean Family}

\author{Kaimin Cheng}
\address{School of Mathematical Sciences, China West Normal University, Nanchong 637002, P. R. China}
\email{ckm20@126.com}


\begin{document}

\begin{abstract}
Let $q=2^e$ with $e\ge2$ even, put $d=(q^2+q+1)/3$, and let
$\sigma(X)=X+X^d+X^{dq}$ be the permutation of $\mathbb F_{q^2}$ introduced by
Ding, Qu, Wang, Yuan, and Yuan. For $\alpha\in\mathbb F_q^*$, define the
Boolean function
\[
f_\alpha(x)=\operatorname{Tr}_{q^2}\bigl(\alpha(\sigma^{-1}(x))^3\bigr),
\qquad x\in\mathbb F_{q^2}.
\]
In this paper, we determine the complete Walsh distribution of $f_\alpha$ in the remaining
cubic case $\alpha\in(\mathbb F_q^*)^3$. More precisely, these functions are
not bent but are $2$-plateaued: their Walsh values are precisely $0$ and
$\pm 2q$, with exact multiplicities. The main new tool is a completion
method for the outside Walsh coefficients: the punctured Fourier transform
arising from the outside reduction is filled on the missing line, a modification
invisible to outside frequencies, and the completed function is then identified
with a Boolean component of a Kasami APN monomial. The APN property supplies a
fourth-moment identity which, together with the known subfield spectrum and a
Hasse divisibility congruence, forces the pointwise cubic spectrum.
\end{abstract}

\maketitle

\section{Introduction}

\subsection{Background and problem}

Let $q=2^e$ with $e\ge2$ even.  Write $\Tr_q$ and $\Tr_{q^2}$ for the absolute trace maps from $\Fq$ and $\Fqq$ to $\Ftwo$, respectively.  For a Boolean function $f:\Fqq\to\Ftwo$, its Walsh transform is
\[
  \W_f(\beta)=\sum_{x\in\Fqq}(-1)^{f(x)+\Tr_{q^2}(\beta x)},\qquad \beta\in\Fqq.
\]
The function $f$ is bent if $\abs{\W_f(\beta)}=q$ for all $\beta\in\Fqq$.  More generally, the full Walsh distribution records the correlation of $(-1)^f$ with all additive characters.  It determines the nonlinearity and plateauedness of $f$, and it is the basic spectral datum behind the correlation and weight distributions arising from the associated trace-code and sequence constructions.  Bent and plateaued functions are therefore central objects in Boolean function theory, coding theory, cryptography, and sequence design; standard references include \cite{CarletMesnager2016,Mesnager2016}.

Put $d=(q^2+q+1)/3$.  Since $q\equiv1\pmod 3$, this is an integer.  Ding, Qu, Wang, Yuan, and Yuan  \cite{Ding2015} proved that
\[  \sigma(X)=X+X^d+X^{dq}
\]
permutes $\Fqq$. For each $\alpha\in\Fq^*$ define
\begin{equation}\label{eq:falpha-def}
  f_\alpha(x)=\Tr_{q^2}\bigl(\alpha(\sigma^{-1}(x))^3\bigr),\qquad x\in\Fqq.
\end{equation}
Li, Li, Helleseth, and Qu conjectured that $f_\alpha$ is bent if and only if $\alpha$ is not a cube in $\Fq$ \cite{Li2023}.  This conjecture was proved in \cite{Cheng2026}.

The present paper is a sequel to \cite{Cheng2026}.  It does not replace the proof of the bentness conjecture there; rather, it determines the complete Walsh distribution in the remaining cubic case, namely when $\alpha\in(\Fq^*)^3$.

The proof in \cite{Cheng2026} naturally splits the Walsh parameter $\beta$ into the subfield regime $\beta\in\Fq$ and the outside regime $\beta\in\Fqq\setminus\Fq$.  The subfield spectrum is explicit.  If $\alpha$ is not a cube, then $\W_{f_\alpha}(\beta)=q$ for all $\beta\in\Fq$.  If $\alpha$ is a cube, then the subfield values are $-2q$ with multiplicity $q/4$ and $2q$ with multiplicity $3q/4$.  Hence cubic $\alpha$ are already excluded from bentness on the subfield.

The outside regime is subtler.  The paper \cite{Cheng2026} reduced the outside coefficients to a two-variable exponential sum over $\Fq$ and proved a binary Hasse congruence modulo $2q$.  In the noncubic case, this congruence says that every outside coefficient is an odd multiple of $q$; the second Walsh moment then forces the values to be exactly $\pm q$.  In the cubic case, however, the same congruence only gives
\begin{equation}\label{eq:cubic-congruence-intro}
  \W_{f_\alpha}(\beta)\equiv0\pmod{2q},
  \qquad \beta\in\Fqq\setminus\Fq.
\end{equation}
This divisibility does not exclude values such as $\pm4q$.  The purpose of the present paper is to prove the exact outside value set and the exact multiplicities, and hence to complete the Walsh distribution in the cubic case.

\subsection{Auxiliary results and the Kasami completion method}

We shall use four inputs from \cite{Cheng2026}: the explicit subfield spectrum, the outside reduction to a two-variable exponential sum, the Hasse divisibility congruence that implies \eqref{eq:cubic-congruence-intro} for cubic parameters, and the already determined noncubic spectrum.  The last input is used only once, to calibrate the two noncubic coefficient classes in a fourth-moment calculation.  The new part of the present paper is the completion and APN-moment argument described next.

The new ingredient is a structural completion of the outside Fourier sum.  The outside reduction writes $\W_{f_\alpha}(\beta)$ as a Fourier transform on $\Fq^2$ with the line $s=0$ deleted.  Since outside parameters have a nonzero frequency in the second coordinate, adding a function supported on that line does not change the relevant Fourier coefficient.  We therefore fill the missing line by the constant value $1$.  After the coordinate identification $X=t+s\eta$ for a suitable Artin--Schreier
element $\eta\in\Fqq$, the completed function becomes a component of the
classical Kasami monomial \cite{CKM2021}
\[
  X\longmapsto X^K,
  \qquad
  K=2^{2(e+1)}-2^{e+1}+1.
\]

The proof deliberately uses only the almost perfect nonlinear property of this Kasami power, not a full Walsh-spectrum theorem for its components. This distinction is important: the completion fills the missing line in the
outside Fourier transform and thereby connects the outside Walsh sums with
a vectorial APN object; the fourth-moment argument then supplies exactly
the global information that the Hasse congruence alone lacks. The noncubic spectrum from \cite{Cheng2026} calibrates the two noncubic coefficient classes.  The remaining class is the cubic class, and the fourth moment for cubic coefficients follows.  Finally, the subfield spectrum, Walsh orthogonality, and the divisibility \eqref{eq:cubic-congruence-intro} force every outside cubic coefficient to be $0$ or $\pm2q$.

\subsection{Main results}

\begin{theorem}\label{thm:cubic-outside}
Let $q=2^e$ with $e\ge2$ even, let $f_\alpha$ be defined by \eqref{eq:falpha-def}, and assume that $\alpha$ is a cube in $\Fq^*$.  Then
\[
\set{\W_{f_\alpha}(\beta):\beta\in\Fqq\setminus\Fq}=
\begin{cases}
\set{0},&e=2,\\
\set{-2q,0,2q},&e\ge4.
\end{cases}
\]
For $e\ge4$, the values $2q$ and $-2q$ each occur on $\Fqq\setminus\Fq$ with multiplicity $q(q-4)/8$, while $0$ occurs with multiplicity $3q^2/4$.
\end{theorem}

\begin{remark}\label{rem:relation-to-Cheng}
The outside spectrum for noncubic $\alpha$ was already determined in \cite[Theorem~1.3 and Corollary~1.4]{Cheng2026}: the outside value set is $\{\pm q\}$, with the corresponding multiplicities.  In the present paper this theorem is used only to calibrate the two noncubic classes in the fourth-moment argument.
\end{remark}

Combining Theorem~\ref{thm:cubic-outside} with the known subfield spectrum gives the full cubic distribution.  We use the standard convention that a Boolean function on $\F_{2^n}$ is $s$-plateaued if its Walsh coefficients belong to $\{0,\pm 2^{(n+s)/2}\}$; in the present setting $n=2e$, so $2$-plateaued means that the only possible nonzero Walsh values have absolute value $2q$.

\begin{corollary}\label{cor:full-cubic}
If $\alpha$ is a cube in $\Fq^*$, then $f_\alpha$ is $2$-plateaued on $\Fqq$.  Its full Walsh distribution is
\[
  2q:\frac{q(q+2)}8,
  \qquad
  -2q:\frac{q(q-2)}8,
  \qquad
  0:\frac{3q^2}{4}.
\]
Consequently, its nonlinearity is $q^2/2-q$.
\end{corollary}

The paper is organized as follows. Section~2 records the APN result and the fourth-moment formalism.  Section~3 proves the completion identity and compares the completed Kasami components with the Walsh coefficients of $f_\alpha$.  Section~4 derives the cubic value set and multiplicities from the fourth moment, the subfield spectrum, and the Hasse divisibility.  Section~5 concludes the paper.

\section{Kasami APN input and fourth moments}

\subsection{APN and the fourth moment}

Let $L=\F_{2^n}$ and let $F:L\to L$.  Following the standard terminology of differentially uniform mappings and APN functions \cite{Nyberg1994,CKM2021}, we say that $F$ is APN if, for every $u\in L^*$ and every $v\in L$, the equation
\[
  F(X+u)+F(X)=v
\]
has at most two solutions in $L$.  Since the solutions occur in pairs $X,X+u$ in characteristic two, this is equivalent to saying that the equation has either zero or two solutions.

\begin{proposition}\label{prop:Kasami-APN}
Let $n$ be even, let $\gcd(r,n)=1$, and put
\[
K_{n,r}=2^{2r}-2^r+1.
\]
Then $X\mapsto X^{K_{n,r}}$ is APN on $\F_{2^n}$.
\end{proposition}

\begin{proof}
This is the even-dimensional Kasami APN argument of Carlet, Kim, and Mesnager \cite[Sec.~3.2, Lem.~4 and the paragraph following it]{CKM2021}.  We recall the deduction to make the precise form used here explicit.  Let
\[
  F(X)=X^{2^{2r}-2^r+1}.
\]
Since $n$ is even and $\gcd(r,n)=1$, the integer $r$ is odd.  Let
\[
  T_r(Z)=\sum_{i=0}^{r-1}Z^{2^i},
  \qquad
  f_{r,2^r+1}(Z)=\frac{T_r(Z)^{2^r+1}}{Z^{2^r}}
\]
be the M\"uller--Cohen--Matthews polynomial used in \cite[Sec.~3.2]{CKM2021}.  Since $T_r(Z)$ is divisible by $Z$, the displayed quotient represents a polynomial map; its value at $Z=0$ is understood through this polynomial representative.  By \cite[Lem.~4]{CKM2021}, this polynomial is a permutation of $\F_{2^n}$.  The paragraph following that lemma gives the identity
\[
  F(X)+F(X+1)+1=f_{r,2^r+1}(X+X^2).
\]
The linearized map $X\mapsto X+X^2$ has kernel $\F_2$, and hence every fiber has size either zero or two.  Composing with the permutation $f_{r,2^r+1}$ and adding the constant $1$ do not increase fiber sizes.  Thus the derivative $X\mapsto F(X+1)+F(X)$ has at most two solutions for every right-hand side.  For arbitrary $u\in\F_{2^n}^*$,
\[
  F(X+u)+F(X)
  =u^{K_{n,r}}\left(F\left(\frac Xu+1\right)+F\left(\frac Xu\right)\right).
\]
Therefore every nonzero derivative of $F$ has fibers of size at most two.  Hence $F$ is APN.
\end{proof}

For $A,B\in L$, define
\[
  S_A(B)=\sum_{X\in L}(-1)^{\Tr_L(AF(X)+BX)},
\]
where $\Tr_L$ is the absolute trace from $L$ to $\Ftwo$.

\begin{lemma}\label{lem:APN-fourth-moment}
Let $F:L\to L$ be APN and put $N=\abs L$.  Then
\[
  \sum_{A\in L^*,\,B\in L}S_A(B)^4=2N^3(N-1).
\]
\end{lemma}

\begin{proof}
First sum over all $A\in L$ and $B\in L$.  Let
\[
  \psi(z)=(-1)^{\Tr_L(z)}
\]
be the canonical additive character of $L$.  Then
\[
  S_A(B)=\sum_{X\in L}\psi(AF(X)+BX).
\]
Writing $\mathbf X=(X_1,X_2,X_3,X_4)$, expanding the fourth power gives
\[
\begin{aligned}
\sum_{A,B\in L}S_A(B)^4
&=\sum_{A,B\in L}
  \sum_{\mathbf X\in L^4}
  \psi\left(
      A\sum_{i=1}^4 F(X_i)
      +B\sum_{i=1}^4 X_i
  \right)  \\
&=\sum_{\mathbf X\in L^4}
  \left(
    \sum_{A\in L}
    \psi\left(A\sum_{i=1}^4F(X_i)\right)
  \right)
  \left(
    \sum_{B\in L}
    \psi\left(B\sum_{i=1}^4X_i\right)
  \right).
\end{aligned}
\]
By the orthogonality of additive characters,
\[
  \sum_{A\in L}\psi(AU)=
  \begin{cases}
  N,& U=0,\\
  0,& U\ne0,
  \end{cases}
\]
and the same formula holds for the sum over $B$.  Hence a quadruple
$(X_1,X_2,X_3,X_4)\in L^4$ contributes $N^2$ precisely when
\[
  \sum_{i=1}^4X_i=0
  \qquad\text{and}\qquad
  \sum_{i=1}^4F(X_i)=0,
\]
and contributes $0$ otherwise.  Since the characteristic is two, these two
conditions are exactly
\[
  X_1+X_2+X_3+X_4=0,
  \qquad
  F(X_1)+F(X_2)+F(X_3)+F(X_4)=0.
\]
If $T$ denotes the number of ordered quadruples satisfying these two
conditions, then
\[
  \sum_{A,B\in L}S_A(B)^4=N^2T.
\]
We now count $T$.  Put $u=X_1+X_2=X_3+X_4$.  If $u=0$, then $X_2=X_1$ and $X_4=X_3$, giving $N^2$ quadruples.  If $u\ne0$, then for each choice of $u$ and $X_1$, the derivative value $F(X_1+u)+F(X_1)$ is fixed.  The equation $F(X_3+u)+F(X_3)=F(X_1+u)+F(X_1)$ has the two solutions $X_3=X_1$ and $X_3=X_1+u$, and the APN property gives no further solutions.  Thus exactly two values of $X_3$ are possible, and then $X_2$ and $X_4$ are determined.  Hence
\[
  T=N^2+2N(N-1)=3N^2-2N.
\]
Thus
\[
  \sum_{A,B\in L}S_A(B)^4=N^2(3N^2-2N).
\]
The contribution of $A=0$ is $N^4$, since $S_0(0)=N$ and $S_0(B)=0$ for $B\ne0$.  Subtracting this contribution gives the claimed identity.
\end{proof}

\subsection{Cube classes for the present Kasami exponent}

From now on in this section let $L=\Fqq$, so $\abs L=q^2$, and put
\[
  K=2^{2(e+1)}-2^{e+1}+1=4q^2-2q+1.
\]
This is the exponent $K_{2e,e+1}$, and $\gcd(e+1,2e)=1$.

\begin{lemma}\label{lem:gcd-K}
One has $\gcd(K,q^2-1)=3$.  Consequently, the image of the map $X\mapsto X^K$ on $\Fqq^*$ is the subgroup $(\Fqq^*)^3$ of cubes.
\end{lemma}

\begin{proof}
Since $e$ is even, $e+1$ is odd.  Thus $2^{e+1}\equiv-1\pmod3$, and so $3\mid K$.  Also $3\mid q^2-1$.

Let $\ell$ be a prime divisor of $\gcd(K,q^2-1)$.  Then $2^{2e}\equiv 1\pmod \ell$, and putting
$a=2^{e+1}$,
we have $a^2-a+1\equiv0\pmod \ell$. Suppose first that $\ell\ne3$.  Then $a\not\equiv -1\pmod\ell$, since
$a\equiv-1\pmod\ell$ would give
\[
  a^2-a+1\equiv 1+1+1\equiv3\pmod\ell,
\]
forcing $\ell=3$. Multiplying
$a^2-a+1\equiv0\pmod\ell$ by $a+1$ gives
$a^3+1\equiv0\pmod\ell$,
and hence
\[
a^3\equiv -1\pmod\ell,\qquad a^6\equiv1\pmod\ell.
\]
Since $a\not\equiv\pm1\pmod\ell$, $a$ has order exactly
$6$ modulo $\ell$.

Let $m=\operatorname{ord}_\ell(2)$ be the order of $2$ modulo $\ell$. From $2^{2e}\equiv1\pmod\ell$,
we have $m\mid 2e$.  Since $\gcd(e+1,2e)=1$, it follows that
$\gcd(m,e+1)=1$.  Therefore
\[  \operatorname{ord}_\ell(a)=\operatorname{ord}_\ell(2^{e+1})
  =\frac{m}{\gcd(m,e+1)}
  =m,
\]
which gives $m=6$. Hence $\ell$
divides $2^6-1=63$ and $2$ has order $6$ modulo $\ell$.  The only prime
divisor of $63$ different from $3$ is $7$, but $2$ has order $3$ modulo
$7$. This is impossible. Thus the only possible prime divisor of
$\gcd(K,q^2-1)$ is $3$.

It remains to exclude a factor $9$.  Since $e+1$ is odd, $2^{e+1}\equiv2,5,$ or $8\pmod9$, and in all three cases
\[
  (2^{e+1})^2-2^{e+1}+1\equiv3\pmod9.
\]
Thus $9\nmid K$.  Therefore $\gcd(K,q^2-1)=3$.
\end{proof}

\begin{lemma}\label{lem:cube-class-subfield}
For $A\in\Fq^*$, the element $A$ is a cube in $\Fq^*$ if and only if it is a cube in $\Fqq^*$.
\end{lemma}

\begin{proof}
Since $e$ is even, we have $q=2^e\equiv1\pmod 3$.  Hence $3\mid q-1$ and
also $3\mid q^2-1$.  Recall that, in a finite cyclic group $G$ whose order is
divisible by $3$, an element $g\in G$ is a cube if and only if $g^{|G|/3}=1$.
Indeed, if $G=\langle \omega\rangle$ has order $m$ and $g=\omega^j$, then
$g$ is a cube if and only if $j\equiv0\pmod{\gcd(3,m)}$, that is, if and only
if $j\equiv0\pmod3$; this is equivalent to
$g^{m/3}=\omega^{jm/3}=1$.

Applying this criterion first in $\Fq^*$ and then in $\Fqq^*$, we have
\[
  A\in(\Fq^*)^3
  \quad\Longleftrightarrow\quad
  A^{(q-1)/3}=1,
\]
and
\[
  A\in(\Fqq^*)^3
  \quad\Longleftrightarrow\quad
  A^{(q^2-1)/3}=1.
\]
Since $A\in\Fq^*$, we may write
\[
  A^{(q^2-1)/3}
  =A^{(q-1)(q+1)/3}
  =\left(A^{(q-1)/3}\right)^{q+1}.
\]
Put
$\zeta=A^{(q-1)/3}$.
Then $\zeta^3=A^{q-1}=1$, so $\zeta$ is a third root of unity.  In
characteristic two the group of third roots of unity is $\F_4^*$, and hence
\[
  \zeta\in\F_4^*=\{1,\omega,\omega^2\}.
\]
On this group, raising to the power $q+1$ is the same as raising to the
power $2$, because $q\equiv1\pmod3$ and therefore $q+1\equiv2\pmod3$.
Thus
\[
  A^{(q^2-1)/3}=\zeta^{q+1}=\zeta^2.
\]
Since $\F_4^*$ has odd order $3$, the equality $\zeta^2=1$ holds if and only
if $\zeta=1$.  Consequently
\[
  A^{(q^2-1)/3}=1
  \quad\Longleftrightarrow\quad
  \zeta=1
  \quad\Longleftrightarrow\quad
  A^{(q-1)/3}=1.
\]
By the two cube criteria above, this proves that $A$ is a cube in $\Fq^*$
if and only if it is a cube in $\Fqq^*$.
\end{proof}
\begin{lemma}\label{lem:class-moment}
For $A\in\Fqq^*$, the fourth moment
\[
  M_A=\sum_{B\in\Fqq}S_A(B)^4,
  \qquad
  S_A(B)=\sum_{X\in\Fqq}(-1)^{\Tr_{q^2}(AX^K+BX)},
\]
depends only on the cubic class of $A$ in $\Fqq^*/(\Fqq^*)^3$.
\end{lemma}

\begin{proof}
If $A'=A\lambda^K$ with $\lambda\in\Fqq^*$, then the change of variable $Y=\lambda X$ gives
\[
  S_{A'}(B)=S_A(B\lambda^{-1}).
\]
By Lemma~\ref{lem:gcd-K}, the set of $K$-th powers in $\Fqq^*$ is exactly the set of cubes.  Therefore $A$ and $A'$ in the same cubic class have the same multiset of transform values, and hence the same fourth moment.
\end{proof}

\section{Completion of the outside Walsh sums}

\subsection{Reductions from the permutation-inverse family}

We first record the two facts from \cite{Cheng2026} used below.

\begin{proposition}[{\cite[Thm.~1.2]{Cheng2026}}]\label{prop:subfield-spectrum}
Let $\alpha\in\Fq^*$.
\begin{enumerate}
\item If $\alpha$ is not a cube in $\Fq$, then $\W_{f_\alpha}(\beta)=q$ for every $\beta\in\Fq$.
\item If $\alpha$ is a cube in $\Fq$, then, on $\Fq$, the value $-2q$ occurs with multiplicity $q/4$, and the value $2q$ occurs with multiplicity $3q/4$.
\end{enumerate}
\end{proposition}

Fix $\lambda\in\Fq$ with $\Tr_q(\lambda)=1$, and put
\[
  A_x=x^2+x+\lambda,
  \qquad
  \Phi(x)=A_x+A_x^{-1}+A_x^{-2}
  \qquad(x\in\Fq).
\]
The element $A_x$ is never zero, because $\Tr_q(x^2+x)=0$ while $\Tr_q(\lambda)=1$.

\begin{proposition}[{\cite[Prop.~2.3 and its proof]{Cheng2026}}]\label{prop:outside-reduction}
Let $\beta\in\Fqq\setminus\Fq$ and put $b=\beta+\beta^q\in\Fq^*$. Choose $\theta\in\Fqq$ such that $\theta^2+\theta=\lambda+1$ and $\theta^q=\theta+1$, and write uniquely
\[
  \beta=b(c+\theta),\qquad c\in\Fq.
\]
Then
\begin{equation}\label{eq:outside-reduction}
  \W_{f_\alpha}(\beta)+q
  =
  \sum_{x,u\in\Fq}
  \chiq{\alpha b^{-3}\Phi(x)u^3+(c+x+1)u},
\end{equation}
where $\chi_q(z)=(-1)^{\Tr_q(z)}$.
\end{proposition}

\subsection{Puncturing and completion}

For $\delta\in\Fq^*$ define a function $G_\delta$ on $\Fq^2$ by
\[
  G_\delta(s,t)=
  \begin{cases}
  \chiq{\delta s^3\Phi(t/s)},&s\ne0,\\
  0,&s=0.
  \end{cases}
\]
For a function $H$ on $\Fq^2$, write
\[
  \widehat H(a_1,a_2)=\sum_{s,t\in\Fq}H(s,t)\chiq{a_1s+a_2t}.
\]

\begin{lemma}[Punctured Fourier form]\label{lem:punctured-Fourier}
With the notation of Proposition~\ref{prop:outside-reduction},
\[
  \W_{f_\alpha}(\beta)=\widehat G_{\alpha b^{-3}}(c+1,1).
\]
\end{lemma}

\begin{proof}
In \eqref{eq:outside-reduction}, the slice $u=0$ contributes $q$.  Removing it gives
\[
  \W_{f_\alpha}(\beta)
  =\sum_{u\in\Fq^*,\,x\in\Fq}
  \chiq{\alpha b^{-3}\Phi(x)u^3+(c+x+1)u}.
\]
Put $s=u$ and $t=ux$.  Then $x=t/s$, and the last display becomes
\[
  \sum_{s\in\Fq^*,\,t\in\Fq}
  \chiq{\alpha b^{-3}s^3\Phi(t/s)+(c+1)s+t}
  =\widehat G_{\alpha b^{-3}}(c+1,1).
\]
\end{proof}

Now fill the missing line $s=0$ by the constant value $1$:
\[
  B_\delta(s,t)=
  \begin{cases}
  \chiq{\delta s^3\Phi(t/s)},&s\ne0,\\
  1,&s=0.
  \end{cases}
\]
This changes no outside coefficient, since
\[
  \widehat B_\delta(c+1,1)-\widehat G_\delta(c+1,1)
  =\sum_{t\in\Fq}\chiq{t}=0.
\]

Choose $\eta\in\Fqq$ satisfying $\eta^2+\eta=\lambda$ and $\eta^q=\eta+1$.  Every $X\in\Fqq$ has a unique expression
\[
  X=t+s\eta,
  \qquad s,t\in\Fq.
\]

\begin{lemma}\label{lem:Kasami-completion}
For every $\delta\in\Fq^*$ and every $s,t\in\Fq$, with $X=t+s\eta$, one has
\begin{equation}\label{eq:Kasami-completion}
  B_\delta(s,t)=(-1)^{\Tr_{q^2}(\delta X^K)}.
\end{equation}
\end{lemma}

\begin{proof}
First assume $s\ne0$.  Write $\bar X=X^q$ and put $r=X/\bar X$.  Since $\bar X=t+s(\eta+1)$, we have $X+\bar X=s$.  Hence
\[
  A_{t/s}
  =\frac{t^2+st+\lambda s^2}{s^2}
  =\frac{X\bar X}{(X+\bar X)^2}
  =\frac{r}{(r+1)^2}.
\]
Also $s=(r+1)\bar X$.  Since
\[
  \Phi(t/s)=A_{t/s}+A_{t/s}^{-1}+A_{t/s}^{-2},
\]
we obtain
\[
\begin{aligned}
  s^3\Phi(t/s)
  &=(r+1)^3\bar X^3\left(\frac r{(r+1)^2}
     +\frac{(r+1)^2}{r}+\frac{(r+1)^4}{r^2}\right)  \\
  &=\bar X^3\left(r(r+1)+\frac{(r+1)^5}{r}
     +\frac{(r+1)^7}{r^2}\right).
\end{aligned}
\]
In characteristic two,
\[
  r(r+1)+\frac{(r+1)^5}{r}+\frac{(r+1)^7}{r^2}
  =r^5+r^{-2}.
\]
Therefore
\[
  s^3\Phi(t/s)=X^5\bar X^{-2}+\bar X^5X^{-2}.
\]
Multiplying by $\delta\in\Fq$, we get
\[
  \delta s^3\Phi(t/s)
  =
  \delta X^5\bar X^{-2}+\delta\bar X^5X^{-2}.
\]
The two terms on the right are conjugate over $\Fq$. Indeed, if
$Y=\delta X^5\bar X^{-2}$, then, since $\delta\in\Fq$ and $\bar X=X^q$, one has
\[
  Y^q=\delta\bar X^5X^{-2}.
\]
Hence
\[
  \delta s^3\Phi(t/s)=Y+Y^q
  =\Tr_{\Fqq/\Fq}(Y).
\]
By transitivity of trace,
\[
  \Tr_q\bigl(\delta s^3\Phi(t/s)\bigr)
  =
  \Tr_q\bigl(Y+Y^q\bigr)
  =
  \Tr_{q^2}(Y)
  =
  \Tr_{q^2}\bigl(\delta X^5\bar X^{-2}\bigr).
\]
Since $\bar X=X^q$ and $X\ne0$, we have $X^5\bar X^{-2}=X^{5-2q}$.  Finally,
\[
  K=4q^2-2q+1\equiv5-2q\pmod{q^2-1}.
\]
This proves \eqref{eq:Kasami-completion} for $s\ne0$.

If $s=0$, then $X=t\in\Fq$.  For $t=0$, both sides of \eqref{eq:Kasami-completion} are $1$.  For $t\ne0$, the congruence $K\equiv3\pmod{q-1}$ gives $\delta X^K=\delta t^3\in\Fq$, and hence
\[
  \Tr_{q^2}(\delta X^K)=\Tr_q(2\delta t^3)=0.
\]
Thus the right side of \eqref{eq:Kasami-completion} is again $1$, which agrees with the definition of $B_\delta$ on $s=0$.
\end{proof}

\begin{lemma}\label{lem:linear-term}
For $a,b,s,t\in\Fq$ and $X=t+s\eta$, one has
\begin{equation}\label{eq:linear-term}
  \Tr_{q^2}\bigl((a+b\eta)X\bigr)=\Tr_q\bigl((a+b)s+bt\bigr).
\end{equation}
\end{lemma}

\begin{proof}
Since $\eta^2=\eta+\lambda$, the coefficient of $\eta$ in $(a+b\eta)(t+s\eta)$ is $as+bt+bs$.  The relative trace from $\Fqq$ to $\Fq$ is the coefficient of $\eta$, because $\eta^q=\eta+1$.  Applying $\Tr_q$ gives \eqref{eq:linear-term}.
\end{proof}

\begin{proposition}[{\cite[Proof of Theorem~1.2]{Cheng2026}}]
\label{prop:subfield-reduction-formula}
Let $\alpha\in\Fq^*$ and $a\in\Fq$.  With
\[
  A_x=x^2+x+\lambda\qquad (x\in\Fq),
\]
one has
\begin{equation}\label{eq:subfield-reduction-formula}
  \W_{f_\alpha}(a)
  =
  \sum_{x,y\in\Fq}
  \chi_q\bigl(\alpha y^3A_x+ay\bigr).
\end{equation}
Moreover, if
\[
  \mathcal T=\{A\in\Fq^*:\Tr_q(A)=1\},
  \qquad
\Psi(A)=A+A^{-1}+A^{-2},
\]
then the map $x\mapsto A_x$ is two-to-one from $\Fq$ onto $\mathcal T$,
and $\Psi$ permutes $\mathcal T$.
\end{proposition}

\begin{lemma}\label{lem:subfield-comparison}
For every $a\in\Fq$ and every $\alpha\in\Fq^*$,
\[
  S_\alpha(a)=\W_{f_\alpha}(a),
\]
where
\[
  S_\alpha(\Gamma)=
  \sum_{X\in\Fqq}
  (-1)^{\Tr_{q^2}(\alpha X^K+\Gamma X)}.
\]
\end{lemma}

\begin{proof}
Put
\[
  \mathcal T=\{A\in\Fq^*:\Tr_q(A)=1\},
  \qquad
  \Psi(A)=A+A^{-1}+A^{-2}.
\]
By Lemmas~\ref{lem:Kasami-completion} and~\ref{lem:linear-term}, applied
with $\Gamma=a$, we have
\[
\begin{aligned}
  S_\alpha(a)
  &=
  q+\sum_{s\in\Fq^*,\,t\in\Fq}
  \chi_q\bigl(\alpha s^3\Phi(t/s)+as\bigr)      \\
  &=
  q+\sum_{s\in\Fq^*,\,x\in\Fq}
  \chi_q\bigl(\alpha s^3\Phi(x)+as\bigr),
\end{aligned}
\]
where in the second line we put $x=t/s$.  Since
\[
  \Phi(x)=A_x+A_x^{-1}+A_x^{-2}=\Psi(A_x),
\]
and since, by Proposition~\ref{prop:subfield-reduction-formula}, the map
$x\mapsto A_x$ is two-to-one from $\Fq$ onto $\mathcal T$ while $\Psi$
permutes $\mathcal T$, it follows that
\[
  S_\alpha(a)
  =
  q+2\sum_{s\in\Fq^*,\,A\in\mathcal T}
  \chi_q\bigl(\alpha s^3A+as\bigr).
\]
Using again that $x\mapsto A_x$ is two-to-one from $\Fq$ onto $\mathcal T$,
we also have
\[
\begin{aligned}
  \sum_{x,s\in\Fq}
  \chi_q\bigl(\alpha s^3A_x+as\bigr)
  &=
  q+
  \sum_{s\in\Fq^*,\,x\in\Fq}
  \chi_q\bigl(\alpha s^3A_x+as\bigr)       \\
  &=
  q+2\sum_{s\in\Fq^*,\,A\in\mathcal T}
  \chi_q\bigl(\alpha s^3A+as\bigr).
\end{aligned}
\]
Therefore
\[
  S_\alpha(a)
  =
  \sum_{x,s\in\Fq}
  \chi_q\bigl(\alpha s^3A_x+as\bigr).
\]
Renaming $s$ as $y$ and applying
Proposition~\ref{prop:subfield-reduction-formula}, we get
\[  S_\alpha(a)=\W_{f_\alpha}(a),
\]
as required.
\end{proof}

\begin{proposition}\label{prop:component-Walsh-comparison}
For every $\alpha\in\Fq^*$, the multisets
\[
  \set{S_\alpha(\Gamma):\Gamma\in\Fqq}
  \qquad\text{and}\qquad
  \set{\W_{f_\alpha}(\beta):\beta\in\Fqq}
\]
are equal, where
\[
  S_\alpha(\Gamma)=\sum_{X\in\Fqq}(-1)^{\Tr_{q^2}(\alpha X^K+\Gamma X)}.
\]
More precisely, if $\Gamma=a+b\eta$ with $a,b\in\Fq$, then $S_\alpha(\Gamma)=\W_{f_\alpha}(\beta)$ for $\beta=a+b\theta$, where $\theta^2+\theta=\lambda+1$ and $\theta^q=\theta+1$.
\end{proposition}

\begin{proof}
First suppose $b\ne0$ and put $c=a/b$.  By Lemmas~\ref{lem:Kasami-completion} and~\ref{lem:linear-term}, and by the change of variable $X=b^{-1}Y$ in $\Fqq$, we have
\[
  S_\alpha(a+b\eta)=S_{\alpha b^{-3}}(c+\eta)=\widehat B_{\alpha b^{-3}}(c+1,1).
\]
The line $s=0$ is invisible at second frequency $1$, so the last term equals $\widehat G_{\alpha b^{-3}}(c+1,1)$.  By Lemma~\ref{lem:punctured-Fourier}, this is $\W_{f_\alpha}(b(c+\theta))=\W_{f_\alpha}(a+b\theta)$.

It remains to consider $b=0$. This is exactly Lemma~\ref{lem:subfield-comparison}. The proposition follows.
\end{proof}

\section{Fourth moments and the cubic outside spectrum}

\subsection{Fourth moment for cubic components}

We now combine the completion with the APN fourth moment.  Put $N=q^2=\abs{\Fqq}$.

\begin{proposition}\label{prop:cubic-component-fourth}
If $A\in(\Fqq^*)^3$, then
\[
  \sum_{B\in\Fqq}S_A(B)^4=4q^6.
\]
If $A\notin(\Fqq^*)^3$, then
\[
  \sum_{B\in\Fqq}S_A(B)^4=q^6.
\]
\end{proposition}

\begin{proof}
By Proposition~\ref{prop:Kasami-APN}, the map $X\mapsto X^K$ is APN on
$\Fqq$. Applying Lemma \ref{lem:APN-fourth-moment} with $L=\Fqq$ gives
\begin{equation}\label{eq:total-Kasami-fourth}
\sum_{A\in\Fqq^*,\,B\in\Fqq}S_A(B)^4=2N^3(N-1).
\end{equation}
By Lemma~\ref{lem:class-moment}, the fourth moment depends only on the cubic class of $A$.  We first determine the two noncubic class moments. Lemma~\ref{lem:cube-class-subfield} implies that the inclusion
$\Fq^*\hookrightarrow\Fqq^*$ preserves and reflects the property of being a
cube.  We claim that the induced map
\[
  \Fq^*/(\Fq^*)^3\longrightarrow \Fqq^*/(\Fqq^*)^3
\]
is injective.  Indeed, suppose that two elements $A,B\in\Fq^*$ have the same
image in $\Fqq^*/(\Fqq^*)^3$.  Then $AB^{-1}\in(\Fqq^*)^3$.
Since $AB^{-1}\in\Fq^*$, Lemma~\ref{lem:cube-class-subfield} gives $AB^{-1}\in(\Fq^*)^3$. Thus $A$ and $B$ already represent the same class in
$\Fq^*/(\Fq^*)^3$.  Hence the induced map is injective. Both quotients have order three, so this map is a bijection.  Consequently each of the two noncube classes of $\Fqq^*$ has a representative $\alpha\in\Fq^*$ which is noncubic in $\Fq$.

For such a representative $\alpha$, \cite[Theorem~1.2 and Theorem~1.3]{Cheng2026} give $\W_{f_\alpha}(\beta)=\pm q$ for every $\beta\in\Fqq$: the subfield values are all $q$, and the outside values are $\pm q$.  Proposition~\ref{prop:component-Walsh-comparison} identifies the multiset of the $S_\alpha(B)$ with the multiset of the $\W_{f_\alpha}(\beta)$.  Hence, for each noncube class,
\[
  M_n=\sum_{B\in\Fqq}S_\alpha(B)^4=q^2q^4=q^6=N^3.
\]
It remains to determine the cubic class moment. Let $M_c$ be the moment for the cube class. Since each coset of $(\Fqq^*)^3$ in $\Fqq^*$ has size $(N-1)/3$, the class decomposition of
\eqref{eq:total-Kasami-fourth} gives
\[
  \frac{N-1}{3}M_c+2\frac{N-1}{3}M_n=2N^3(N-1).
\]
After cancellation, $M_c+2M_n=6N^3$.  Thus $M_c=4N^3=4q^6$.
\end{proof}

\begin{corollary}\label{cor:cubic-W-fourth}
If $\alpha\in\Fq^*$ is a cube, then
\[
  \sum_{\beta\in\Fqq}\W_{f_\alpha}(\beta)^4=4q^6.
\]
\end{corollary}

\begin{proof}
By Lemma~\ref{lem:cube-class-subfield}, $\alpha$ is also a cube in $\Fqq^*$.  Proposition~\ref{prop:cubic-component-fourth} gives the fourth moment for $S_\alpha$, and Proposition~\ref{prop:component-Walsh-comparison} identifies this moment with the fourth moment of $\W_{f_\alpha}$.
\end{proof}

\subsection{The outside values}

We need one final fact from \cite{Cheng2026}, namely the cubic divisibility obtained by combining the outside reduction with the intrinsic Hasse congruence.

\begin{lemma}\label{lem:cubic-divisibility}
If $\alpha\in\Fq^*$ is a cube and $\beta\in\Fqq\setminus\Fq$, then
\[
  \W_{f_\alpha}(\beta)\equiv0\pmod{2q}.
\]
\end{lemma}

\begin{proof}
Put $b=\beta+\beta^q\in\Fq^*$, choose $\theta^2+\theta=\lambda+1$ with $\theta^q=\theta+1$, and write $\beta=b(c+\theta)$ with $c\in\Fq$. Set $\kappa=\alpha b^{-3}$ and $\Lambda=c^2+c+\lambda$. The outside reduction in \cite[Prop.~2.3]{Cheng2026} gives
\[
  \W_{f_\alpha}(\beta)+q=\mathcal F_{\kappa,\Lambda},
\]
where
\[
\mathcal F_{\delta,\Lambda}=
\sum_{h,y\in\Fq}\chi_q\bigl(y(h^3+h^2+\Lambda h)+\delta y^3(h^8+h+\Lambda^4+\Lambda^2+\Lambda)\bigr).
\]
The intrinsic Hasse congruence \cite[Prop.~2.4]{Cheng2026} says that, with $N=q-1$,
\[
  \mathcal F_{\delta,\Lambda}
  \equiv q\bigl(1+\delta^{N/3}+\delta^{2N/3}\bigr)\pmod{2q},
\]
where the parenthesized element of $\F_2$ is identified with $0$ or $1$. Since $b^{-3}$ is a cube and $\alpha$ is a cube, $\kappa$ is a cube in $\Fq^*$. Thus $\kappa^{N/3}=\kappa^{2N/3}=1$, and $1+1+1=1$ in $\F_2$. Hence $\mathcal F_{\kappa,\Lambda}\equiv q\pmod{2q}$, and therefore $\W_{f_\alpha}(\beta)\equiv0\pmod{2q}$, as desired.
\end{proof}

\begin{proof}[Proof of Theorem~\ref{thm:cubic-outside}]
Assume that $\alpha$ is a cube in $\Fq^*$.  By Proposition~\ref{prop:subfield-spectrum}, the subfield fourth moment is
\[
  \sum_{\gamma\in\Fq}\W_{f_\alpha}(\gamma)^4
  =q(2q)^4=16q^5.
\]
Corollary~\ref{cor:cubic-W-fourth} gives the full fourth moment $4q^6$.  Hence the outside fourth moment is
\[
  \sum_{\gamma\in\Fqq\setminus\Fq}\W_{f_\alpha}(\gamma)^4
  =4q^6-16q^5=4q^5(q-4).
\]
Similarly, Walsh Plancherel and Proposition~\ref{prop:subfield-spectrum} give the outside square moment
\[
  \sum_{\gamma\in\Fqq\setminus\Fq}\W_{f_\alpha}(\gamma)^2
  =q^4-q(2q)^2=q^3(q-4).
\]
By Lemma~\ref{lem:cubic-divisibility}, for every outside $\gamma$ we may write
\[
  \W_{f_\alpha}(\gamma)=2qK_\gamma,
  \qquad K_\gamma\in\mathbb Z.
\]
The two moment formulas become
\[
  \sum_{\gamma\notin\Fq}K_\gamma^2=\frac{q(q-4)}4,
  \qquad
  \sum_{\gamma\notin\Fq}K_\gamma^4=\frac{q(q-4)}4.
\]
Therefore
\[
  \sum_{\gamma\notin\Fq}\bigl(K_\gamma^4-K_\gamma^2\bigr)=0.
\]
Note that each summand $K_\gamma^2(K_\gamma^2-1)$ is a nonnegative integer. It follows that $K_\gamma\in\{-1,0,1\}$ for every outside $\gamma$.  Thus all outside values lie in $\{-2q,0,2q\}$.

If $e=2$, then $q=4$ and the outside square moment is zero, so all outside values are zero.  Assume $e\ge4$.  The number of nonzero outside values is
\[
  \frac{q^3(q-4)}{(2q)^2}=\frac{q(q-4)}4.
\]
Since $\sigma(0)=0$, we have $f_\alpha(0)=0$.  Therefore the first Walsh orthogonality relation gives
\[
  \sum_{\gamma\in\Fqq}\W_{f_\alpha}(\gamma)
  =\sum_{x\in\Fqq}(-1)^{f_\alpha(x)}\sum_{\gamma\in\Fqq}(-1)^{\Tr_{q^2}(\gamma x)}
  =q^2.
\]
By Proposition \ref{prop:subfield-spectrum}(2), the subfield contribution is
\[
  \frac q4(-2q)+\frac{3q}{4}(2q)=q^2.
\]
Thus the outside sum is zero, and the values $2q$ and $-2q$ occur equally often.  Each occurs $q(q-4)/8$ times, and the remaining outside multiplicity is
\[
  q^2-q-\frac{q(q-4)}4=\frac{3q^2}{4}.
\]
For $q\ge16$ all three multiplicities are positive, so the outside value set is exactly $\{-2q,0,2q\}$.  The theorem follows.
\end{proof}

\begin{proof}[Proof of Corollary~\ref{cor:full-cubic}]
Add the subfield multiplicities from Proposition~\ref{prop:subfield-spectrum} to the outside multiplicities from Theorem~\ref{thm:cubic-outside}.  The multiplicity of $2q$ is
\[
  \frac{3q}{4}+\frac{q(q-4)}8=\frac{q(q+2)}8,
\]
where the outside term is zero when $q=4$.  The multiplicity of $-2q$ is
\[
  \frac q4+\frac{q(q-4)}8=\frac{q(q-2)}8.
\]
The zero multiplicity is $3q^2/4$.  Since $2q=2^{e+1}=2^{(2e+2)/2}$, the Boolean function is $2$-plateaued.
\end{proof}

\section{Conclusion}

We have determined the complete Walsh distribution of the permutation-inverse family \eqref{eq:falpha-def} for cubic parameters.  The result shows that these functions are exactly $2$-plateaued: the only nonzero Walsh values are $\pm2q$, and their multiplicities are explicit.  Thus the cubic side of the family is now described with the same spectral precision as the noncubic, bent side.

The method is the main point of the proof.  A punctured outside Fourier transform is completed by adding a line that is invisible to outside frequencies.  The completed transform is then identified with a component of a Kasami monomial.  This converts a local divisibility statement, coming from the Hasse congruence, into a global moment calculation governed by the APN property of the Kasami power.  The fourth moment so obtained rules out all larger multiples of $2q$ and leaves only $0$ and $\pm2q$.  This completion-to-APN-moment mechanism may be useful in other Walsh-spectrum problems where a punctured exponential sum is naturally adjacent to a differentially uniform monomial.

\end{document}